\documentclass[12pt,a4paper]{amsart}
\usepackage{a4wide}
\usepackage{amsmath}
\usepackage{amsthm}
\usepackage{amssymb}
\usepackage{hyperref}
\usepackage{amsfonts}
\usepackage{color}
\usepackage{pdfsync}
\usepackage[latin1]{inputenc}
\usepackage{enumerate}
\usepackage{graphicx}
\usepackage{graphics}
\usepackage[french,english]{babel}
\usepackage{amsmath}
\usepackage{a4wide}
\usepackage{amssymb}
\usepackage{amsfonts}
\usepackage[final]{pdfpages}
\usepackage{color}
\usepackage{fancyhdr}
\usepackage{enumerate}
 \usepackage{epsfig}
 \usepackage{pspicture}
 \usepackage{pst-grad}
  \usepackage{pstricks}
 \usepackage{pst-plot}
     \usepackage{eepic}
\newtheorem{Th}{Theorem}

\newtheorem{obs}{Remark}
\newtheorem{lema}{Lemma}
\numberwithin{equation}{section} \numberwithin{Th}{section}
\numberwithin{cor}{section} \numberwithin{lema}{section}
\numberwithin{prop}{section} \numberwithin{obs}{section}
\numberwithin{Def}{section}

\def\rr{\mathbb{R}}

\def\q{\quad}
\def\D{\Delta}

\def\dx{\ \mathrm{dx}}
\def\dr{\ \mathrm{dr}}

\def\ds{\ \mathrm{ds}}

\def\iir0{\iint\limits_{\Omega_{r_0}\times(0,T)}}

\def\sing2{|x|^{-\frac{N+2}{2}+\sqrt{\lambda_\star-\lambda}}}

\def\benu{\begin{enumerate}}
\def\enu{\end{enumerate}}

\def\cat2{\big(\frac{1}{r_0}\big)^{2\lambda}}
\def\n{\nabla}
\def\be{\begin{equation}}
\def\ee{\end{equation}}

\def\nn{\mathbb{N}}

\def\into{\int_{\Omega}}

\def\hoi{H_{0}^{1}(\Omega)}

\def\eps{\varepsilon}
\def\n{\nabla}
\def\D{\Delta}
\def\q{\quad}

\def\eps{\varepsilon}

\def\({\Big(}
\def\){\Big)}
\def\nn{\mathbb{N}}

\def\into{\int_{\Omega}}
\def\hoi{H_{0}^{1}(\Omega)}

\def\eps{\varepsilon}

\def\be{\begin{equation}}
\def\ee{\end{equation}}
\newcommand{\cf}{\emph{cf. }}
\newcommand{\eg}{\emph{e.g. }}

\def \semi{\mathcal{D}^{1,2}(\Omega)}
\def \semiH{\mathcal{D}^{1,2}(\Omega)}

\title[Optimal results for multipolar Hardy inequalities]
{New estimates for the Hardy constants of multipolar Schr\"{o}dinger operators}
\author[Cristian Cazacu]
{Cristian Cazacu $^{1,2}$}
\thanks{$^1$ Department of Mathematics and Informatics, Faculty of Applied Sciences,
 University Politehnica of Bucharest,
Splaiul Independentei 313, Bucharest, 060042, Romania. }
\thanks{ $^2$ ``Simion Stoilow" Institute of Mathematics of the Romanian Academy, Research Group of the Projects PN-II-ID-PCE-2012-4-0021 and PN-II-ID-PCE-2011-3-0075, 21 Calea Grivitei Street, 010702 Bucharest, Romania. \\
 E-mail: cristi\_cazacu2002@yahoo.com.
}
\begin{document}
\maketitle
\begin{center}
\end{center}
\begin{abstract}
In this paper we study the optimization problem
$$\mu^\star(\Omega):=\inf_{u\in \semi}\frac{\into |\n u|^2 \dx}{\into V u^2 \dx}$$
in a suitable functional space $\semi$. Here, $V$ is the multi-singular potential given by $$V:=\sum_{1\leq i<j\leq n} \frac{|a_i-a_j|^2}{|x-a_i|^2|x-a_j|^2}$$ and all the singular poles $a_1, \ldots, a_n$, $n\geq 2$,  arise either in the interior or at the boundary of a  smooth open domain $\Omega\subset \rr^N$, with  $N\geq 3$ or $N \geq 2$, respectively.

For a bounded domain $\Omega$  containing all the singularities in the interior, we prove that $\mu^\star(\Omega)>\mu^\star(\rr^N)$ when $n\geq 3$ and $\mu^\star(\Omega)=\mu^\star(\rr^N)$ when $n=2$ (It is known from \cite{cristi1} that $\mu^\star(\rr^N)=(N-2)^2/n^2)$.

 In the situation when all the poles are located on the boundary we show that $\mu^\star(\Omega)=N^2/n^2$ if $\Omega$ is either a ball, the exterior of a ball or a half-space. Our results do not depend on the distances between the poles. In addition, in the case of boundary singularities we obtain  that $\mu^\star(\Omega)$ is attained in $\hoi$ when $\Omega$ is a ball and $n\geq 3$. Besides, $\mu^\star(\Omega)$ is attained in $\semi$ when $\Omega$ is the exterior of a ball with $N\geq 3$ and $n\geq 3$ whereas in the case of a half-space $\mu^\star(\Omega)$ is attained in $\semi$ when $n\geq 3$.

    We also analyze the critical constants in the so-called \textit{weak} Hardy inequality which characterizes the range of $\mu's$  ensuring  the existence of a lower bound for the spectrum of the Schr\"{o}dinger operator $-\Delta -\mu V$. In the context of both interior and boundary singularities we show that the critical constants in the weak Hardy inequality are $(N-2)^2/(4n-4)$  and $N^2/(4n-4)$, respectively.
\end{abstract}


\newpage
\section{Introduction}

In this paper we are dealing  with the multi-singular inverse-square potential
$$
V=\sum_{1\leq i<j\leq n} \frac{|a_i-a_j|^2}{|x-a_i|^2|x-a_j|^2}, \ n\geq 2,
$$
 with the  singular poles $a_1, \ldots, a_n$ ($a_i \neq a_j$ for any $i\neq j$). We consider two extremal situations concerning the location of the singularities: all the poles $a_1, \ldots, a_n$ are located either in the interior or on the boundary  of an open smooth domain $\Omega\subset\rr^N$, $N\geq 2$. In the case of interior singularities we assume the dimension restriction $N\geq 3$ (below we give a motivation to this fact).

For such $V$ we consider  multipolar Schr\"{o}dinger operators of the form
     \be\label{opSch}
     L_\mu:=-\D -\mu V, \quad \mu> 0,
     \ee
in which a positive potential is subtracted making the operator $L_\mu$ less positive than the Laplace operator $-\Delta$.

  We are interested to investigate the range of parameters $\mu's$ for which $L_\mu$ is still nonnegative, i.e. the operator inequality $L_\mu\geq 0$ is verified in $L^2$-quadratic forms. This issue is equivalent to the study of the Hardy inequality
\be\label{Hardy_inequality}
\into |\n u|^2 \dx \geq \mu \into V u^2 \dx, \quad \forall u \in C_0^\infty(\Omega).
\ee
The classical Hardy inequality involves a one singular inverse-square potential and it asserts that (see \eg \cite{hardy-polya}) for an interior pole $a_i\in \Omega$  it holds
 \be\label{classical_Hardy_inequality}
 \into |\n u|^2 \dx \geq \lambda \into \frac{u^2}{|x-a_i|^2} \dx, \quad \forall u\in C_0^\infty(\Omega),
 \ee
 if and only if $\lambda\leq \lambda^\star:=(N-2)^2/4$.

  Formally, inequality \eqref{classical_Hardy_inequality} is trivially true for $\lambda=\lambda^\star$ when $N=2$. However, in this case it occurs that \eqref{classical_Hardy_inequality} is not well-posed because the potential term on the right hand side may not be integrable. Indeed, for an interior pole $a_i\in \Omega$ and a given function $u\in C_0^\infty(\Omega)$ such that $u(a_i)=1$ then $u^2/|x-a_i|^2$ is not locally integrable unless $N>2$. Similar integrability properties are valid for the potential term $Vu^2$  in \eqref{Hardy_inequality} since $\lim_{x\to a_i} V|x-a_i|^2$ is finite, so both potentials $V$ and $1/|x-a_i|^2$ are comparable near each singularity $a_i$.

  For these reasons, when dealing with interior singularities it is correct to study \eqref{Hardy_inequality} in higher dimensions $N\geq 3$. Otherwise, if the poles arise at the boundary, the potential term in \eqref{Hardy_inequality} makes sense for any $N\geq 2$. This latter fact is motivated by an improved version  of \eqref{classical_Hardy_inequality} which is true when the singularity $a_i$ is located on the boundary of $\Omega$.   In this case, up to $L^2$-reminder terms, \eqref{classical_Hardy_inequality} holds in the new range of parameters $\lambda\leq N^2/4$ (see \eg \cite{MR2847469}).

The main goal of  this paper is to study the optimal constant of the Hardy inequality \eqref{Hardy_inequality} which is defined through the optimization problem
\be\label{optimalhardy}
\mu^\star(\Omega; a_1, \ldots, a_n):=\inf_{u\in \semi}\frac{\into |\n u|^2 \dx}{\into V u^2 \dx}.
\ee
When there is no risk of confusion, we will write $\mu^\star(\Omega)$  instead of $\mu^\star(\Omega; a_1, \ldots, a_n)$.

 Here, the functional space $\semi$   is defined naturally  as the completion of $C_{0}^{\infty}(\Omega)$ with respect to the norm
\begin{equation}
\|u\|=\into |\n u|^2 \dx.
\end{equation}
Clearly, $\semi=\hoi$ if $\Omega$ is a domain for which the Poincar\'{e} inequality is verified (such as bounded domains)  otherwise,  $\semi \supsetneq \hoi$. Working in $\semi$ is determinant when studying the attainability of $\mu^\star(\Omega)$ since the proper space where we have to look for minimizers in \eqref{optimalhardy} is precisely $\semi$.

Besides, in the present article we also analyze the range of parameters $\mu's$ that ensures the existence of a lower bound for the spectrum of the operator $L_\mu$  with Dirichlet boundary conditions (denoted by $L_\mu^{D}$). Namely, in the spirit of Davies \cite{MR1366614} and Brezis-Marcus \cite{MR1655516},  we are seeking for those $\mu's$ for which there exists a finite constant $c_\mu\in \rr$ (which might also be negative) such that $L_\mu^{D}\geq  c_\mu$.  This is equivalent to the following Hardy-type inequality:
       \be\label{hpineq}
\into |\n u|^2 \dx -\mu  \into V u^2 \dx\geq c_\mu\into u^2 \dx, \quad \forall u\in C_{0}^{\infty}(\Omega).
       \ee
       For any fixed $\mu>0$ let $c_\mu^\star$ to be the supremum of all constants $c_\mu\in \rr$ satisfying \eqref{hpineq}. If   $c_\mu^\star<0$ then \eqref{hpineq} is the so-called  \textit{weak} Hardy inequality.
In view of these, let us also define the  \textit{weak} Hardy constant for inequality \eqref{hpineq}, that is
     \be\label{weakoptimalHardy}
     \mu_\omega^\star(\Omega):=\sup\{ \mu >0 \ |\ \exists \ c_\mu\in \rr \textrm{ satisfying } \eqref{hpineq}\}.
     \ee
        Since the definition of  $\mu^\star(\Omega)$  in \eqref{optimalhardy} is more restrictive, it is not difficult to observe that $\mu^\star(\Omega)$ is smaller than or equal to the weak Hardy constant $\mu_\omega^\star(\Omega)$  in \eqref{weakoptimalHardy}.

  Our interest in optimal Hardy constants is particularly motivated by the previous related results in  Bosi-Dolbeault-Esteban \cite{MR2379440} and  Zuazua and myself in  \cite{cristi1}. First it was shown in \cite{MR2379440} that, in the whole space $\rr^N$, $N\geq 3$, it is true that $$L_\mu^D> 0, \quad  \forall \mu\leq \frac{(N-2)^2}{4n^2},$$
 and the location of the singular poles does not matter.
 More precisely, the authors in    \cite{MR2379440} proved that
 \begin{align}\label{pushu}
 \int_{\rr^N} |\n u|^2 \dx \geq \frac{(N-2)^2}{4 n^2}&  \int_{\rr^N} V u^2 \dx\nonumber\\
 &+\frac{(N-2)^2}{4n} \sum_{i=1}^{n}\int_{\rr^N} \frac{u^2}{|x-a_i|^2} \dx, \quad \forall u\in C_0^\infty(\rr^N),
 \end{align}
 for any fixed configuration $a_1, \ldots, a_n\in \rr^N$, with $a_i \neq a_j$ for $i\neq j$.
 It occurs that the constant $(N-2)^2/(4n^2)$ in \eqref{pushu} is not optimal in the sense of \eqref{optimalhardy}. More precisely, we proved in \cite{cristi1} the improved Hardy inequality
\be\label{wholespace}
 \int_{\rr^N} |\n u|^2 \dx \geq \frac{(N-2)^2}{n^2}  \int_{\rr^N} V u^2 \dx, \quad \forall u\in C_{0}^{\infty}(\rr^N),
\ee
   with the optimal constant
   \begin{equation}\label{optimal.global}
   \mu^\star(\rr^N)=\frac{(N-2)^2}{n^2}.
    \end{equation}
 Moreover, $\mu^\star(\rr^N)$ is not attained in $\semiH$.

The Hardy-type inequalities have been intensively studied in the literature. In particular, they play a crucial influence to the spectral properties of singular Hamiltonian operators which arise in mathematical physics or quantum mechanics (\cf Davies \cite{MR1747888},   Laptev et al. \cite{MR2083152}, \cite{laptev}, Krej{\v{c}}i{\v{r}}{\'{\i}}k-Zuazua \cite{MR2679028},  Brezis-Vazquez \cite{MR1605678}, Pinchover et al.   \cite{tintarevpinchover}, \cite{devyver},   Fefferman \cite{fefe}, etc.). In connection with single poles we also mention for instance papers such as \cite{caldiroli2, cras1, DD, MR2966112} and references therein.   For Hamiltonians with multipolar potentials we refer to the papers by Bosi et al. \cite{MR2379440},  Terracini et al.  \cite{MR2241305}, \cite{MR2735078}, \cite{MR2514383}  and the references therein. In particular, important remarks, comparisons and  applications of such Hamiltonians were emphasized in more details in the Introduction of  \cite{cristi1}.

\subsection*{Useful notations}

Before entering into details let us fix some useful notations and notions.

We designate by $\Gamma$ the boundary of $\Omega$. As usual, we denote by $B_r(x_0)$ the ball of radius $r>0$ centered at a point $x_0 \in \rr^N$, and by $B^c_r(x_0)$ the exterior of $B_r(x_0)$. The upper half-space in $\rr^N$ is defined by the set $$\rr_{+}^{N}:=\{x=(x_1, \ldots, x_N)\in \rr^N \ | \ x_N>0  \}$$ which is supported by the hyperplane $x_N=0$.   In a more general sense, a half-space in $\rr^N$ can be defined as a support of an affine hyperplane of the form $\alpha_1 x_1+\ldots+ \alpha_N x_N=b$, with fixed constants $\alpha_i, b\in \rr$, $i\in \{1, \ldots, N\}$.  For simplicity, in this paper we will only refer to  $\rr_+^N$. However, our next results obtained for $\rr_+^N$ could also be extended to any general half-space.
Besides, we will make use of the notations $$d:=\min_{i, j\in \{1, \ldots, n\}} |a_i-a_j|, \quad  d_\Gamma:= \min_{i\in \{1, \ldots, n\}} \rho(a_i), \quad M:=\min\{d, d_\Gamma\},$$ where $\rho(x)=\inf_{y\in \Gamma}|x-y|$ denotes the distance function to the boundary.

Throughout the paper we will also consider the parameter $\eps>0$ to be small enough since it will be aimed to tend to zero.\\

\noindent{\bf Main results.} In this paper we extend and improve the previous results shown in  \cite{cristi1}.

 In the case of interior singularities we obtain the following unexpected results for $\mu^\star(\Omega)$.
\begin{Th}\label{prop1}
  Assume that $\Omega\subset \rr^N$, $N\geq 3$, is an open bounded domain such that  $a_1, \ldots, a_n\in \Omega$, $n\geq 2$. We assert that
  \be\label{equality.constant}
   \left\{ \begin{array}{ll}
                             \mu^\star(\Omega) =  \frac{(N-2)^2}{n^2}, & if\  n=2, \\ [10pt]
     \frac{(N-2)^2}{n^2} < \mu^\star(\Omega)\leq \frac{(N-2)^2}{4n-4}, &  if\  n\geq 3.\\
                            \end{array}\right.
  \ee
  Moreover, if $n=2$ then \eqref{equality.constant} is verified in any open domain $\Omega$ (not necessary bounded).
  \end{Th}
Theorem \ref{prop1} shows in the case of interior singularities that there is a gap between $\mu^\star(\Omega)$ and $\mu^\star(\rr^N)$ when  $\Omega\subset \rr^N$ is bounded, if and only if $n\geq 3$.

The following theorem provides optimal results concerning the weak Hardy constant $\mu_\omega^\star(\Omega)$ in the case of  interior singularities.
 \begin{Th}\label{moregen}
Assume $\Omega\subset \rr^N$, $N\geq 3$ is an open bounded domain  such that $a_1, \ldots, a_n \in \Omega$, $n\geq 2$. Then we successively obtain that
\begin{enumerate}[(i).]
\item \label{item1}  For any parameter $\mu< (N-2)^2/(4n-4)$,
 there exists a finite constant  $c_\mu\in \rr$ which  depends on $\Omega$ and  $a_1, \ldots, a_n$,  so that the inequality
\be\label{good1}
c_\mu\into u^2 \dx+ \into |\n u|^2 \dx \geq \mu  \into V u^2 \dx,
\ee
is verified for any $u \in C_{0}^{\infty}(\Omega)$.
 \item\label{item2} For any $\mu> (N-2)^2/(4n-4)$  and any constant $\lambda\in \rr$ there exists $u_{\lambda, \mu}\in C_{0}^{\infty}(\Omega)$ such that
\be\label{good2}
\lambda \into u_{\lambda, \mu}^2 \dx+ \into |\n u_{\lambda, \mu}|^2 \dx <  \mu  \into V u_{\lambda, \mu}^2 \dx.
\ee
\item\label{item3}  If $\mu^\star(\Omega)< (N-2)^2/(4n-4)$  then $\mu^\star(\Omega)$ is attained in $\hoi$.

\item\label{item4} If $n=2$  then $\mu^\star(\Omega)$  is never attained in $\semi$.
\end{enumerate}
\end{Th}
\begin{obs}
In particular, Theorem \ref{moregen} emphasizes in the case of interior singularities that the upper bound  $(N-2)^2/(4n-4)$ in \eqref{equality.constant} coincides with the weak Hardy constant, i.e.
$$\mu_\omega^\star(\Omega)=\frac{(N-2)^2}{4n-4}.$$
\end{obs}

Next we state the main results we obtain in the case of boundary singularities.

 It is nowadays a popular fact that, when switching from interior to boundary singularities, the Hardy constant $\mu^\star(\Omega)$ increases. In this latter case we underline in the next theorem that  $\mu^\star(\Omega)$ mainly turns on to depend on both the global and local geometry of the domain near the singular poles.

 \begin{Th}\label{t1} Assume $\Omega$ is either a ball, the exterior of a ball, or a  half-space in $\rr^N$, $N\geq 2$,  such that $a_1, \ldots, a_n\in \Gamma$, $n\geq 2$.  Then it holds
               \be\label{cond3}
               \mu^\star(\Omega)=\frac{N^2}{n^2}.
               \ee
If $\Omega$ is a ball, the constant $\mu^\star(\Omega)$ in \eqref{cond3} is attained in $\hoi$ if and only if $n\geq 3$.\\
If $\Omega$ is the exterior of a ball then $\mu^\star(\Omega)$ is attained in $\semi$ when $N\geq 3$ and $n\geq 3$, whereas if $\Omega$ is a half-space in $\rr^N$ then $\mu^\star(\Omega)$ is attained in $\semi$ when $n\geq 3$.
 \end{Th}

  The main novelty of Theorem \ref{t1} with respect to classical Hardy-type inequalities is concerned with the attainability of $\mu^\star(\Omega)$ in some particular cases. This is a surprising property since, as far as we know, the attainability of the optimal constant $\mu^\star(\Omega)$  is not common in the class of Hardy inequalities.

Although Theorem \ref{t1} reflects optimal results, it only applies for a particular class of domains, with a specific geometry (``nice" domains).  In the following theorem we highlight weak Hardy-type inequalities which are independent of the geometry of $\Omega$.

\begin{Th}\label{te2}
Assume $\Omega\subset \rr^N$, $N\geq 2$, is a smooth bounded domain  such that $a_1, \ldots, a_n \in \Gamma$, $n\geq 2$.
We successively have
\begin{enumerate}[i).]
\item\label{item5} It holds that \be\label{bounds.bond}
\frac{(N-2)^2}{n^2} <   \mu^\star(\Omega)\leq  \frac{N^2}{4n-4}.
\ee
\item\label{item6}  For any parameter $\mu< N^2/(4n-4)$,
 there exists a finite constant  $c_\mu\in \rr$ which  depends on $\Omega$ and  $a_1, \ldots, a_n$,  so that the inequality
\be\label{good3}
c_\mu\into u^2 \dx+ \into |\n u|^2 \dx \geq \mu  \into V u^2 \dx,
\ee
is verified for any $u \in C_{0}^{\infty}(\Omega)$.
 \item\label{item7} For any $\mu> N^2/(4n-4)$  and any constant $\lambda\in \rr$ there exists $u_{\lambda, \mu}\in C_{0}^{\infty}(\Omega)$ such that
\be\label{good4}
\lambda \into u_{\lambda, \mu}^2 \dx+ \into |\n u_{\lambda, \mu}|^2 \dx <  \mu  \into V u_{\lambda, \mu}^2 \dx.
\ee
\item\label{item8}  In addition, if $\mu^\star(\Omega)< N^2/(4n-4)$  then $\mu^\star(\Omega)$ is attained in $\hoi$.
\end{enumerate}
\end{Th}
\begin{obs}
Theorem \ref{te2} extends the results of Theorem \ref{moregen} to the case of boundary singularities with a higher weak Hardy constant, i.e.
$$\mu_\omega^\star(\Omega)=\frac{N^2}{4n-4}.$$
In this case  we observe that $\mu_\omega^\star(\Omega)$ is strictly larger than $\mu^\star(\Omega)$ when $n\geq 3$ and $\Omega$ is either a ball or a hyperplane.  To our knowledge, this gap between the Hardy constants  $\mu^\star(\Omega)$ and $\mu_\omega^\star(\Omega)$ is unusual in the class of Hardy inequalities for convex domains.
\end{obs}

\begin{obs}
  Similar as in Theorems \ref{prop1}-\ref{moregen}, the results of Theorems \ref{t1}-\ref{te2} do not depend on the distances between the singular poles but on their number and their positions relative to the boundary.
\end{obs}

The paper is organized as follows.

In Section \ref{AlternativeApproach} we recall a general approach to prove Hardy inequalities which relies on an integral identity with weights. Applying the preliminary results in Section \ref{AlternativeApproach} we prove Theorem \ref{prop1} in Section \ref{int.sing}. Theorem \ref{moregen} is proven in Section \ref{int.sing} combining the local behavior of $V$ near the singular poles,  the Hardy inequality with interior singularities \eqref{classical_Hardy_inequality} and cut-off arguments. In Section \ref{bound.sing} we prove Theorem \ref{t1} by adapting the weights in Section \ref{AlternativeApproach} to both the structure of our multi-singular potential $V$ and the entire geometry of the domain.
Theorem \ref{te2} is shown at the end of Section \ref{bound.sing} following the same lines as the proof of Theorem \ref{moregen} but applying  the local  Hardy inequality with boundary singularities \eqref{classical_Hardy_inequality} instead. In Section \ref{com} we point out some important remarks, comments and  open problems.

\section{A general approach to multipolar Hardy inequalities}\label{AlternativeApproach}
In this preliminary section we discuss a general method to prove multipolar Hardy inequalities based on the famous Agmon-Allegretto-Piepenbrick Theorem which asserts, roughly speaking,  that a Schr\"{o}dinger operator $L$ is nonnegative in a domain $\Omega$ if and only if there exists a positive solution in $\Omega$ of the equation $Lu=0$ (\cf \cite{agmon, MR0374628, MR0342829}).

This standard method was already used in \cite{cristi1} for $\Omega=\rr^N$ in the context of interior singularities.

Let $\Omega\subset \rr^N$ and $a_1, \ldots, a_n\in \overline{\Omega}$, $n\geq 2$. Then we recall that for any $u \in C_0^1(\Omega\setminus \{a_1, a_2,
\ldots a_n\})$ and any $\phi\in C^2(\Omega\setminus\{a_1, \ldots, a_n\})$ with $\phi>0$ in $\Omega$ ($\phi$ is allowed to be singular at any of the poles $a_1, \ldots, a_n$),  the following identity is verified
\begin{equation}\label{gen}
\into \left( |\n u|^2 +\frac{\D \phi}{\phi}u^2\right)\dx =\into \left|\n
u- \frac{\n \phi}{\phi} u\right|^2 \dx =\into \phi^2|\n (u
\phi^{-1})|^2 \dx.
\end{equation}
Identity \eqref{gen} follows easily by integration by parts but it gives rise to  important applications.

To be more specific,  we can remark from \eqref{gen} that  the  Schr\"{o}dinger operator  $L_\mu=-\Delta -\mu V$  is nonnegative in $\Omega$ if there exists a distribution $\phi\in \mathcal{D}'(\Omega)$  satisfying
$$-\frac{\Delta \phi}{\phi}= \mu V , \textrm{ a.e. in }\Omega, \quad \phi> 0 \textrm{ in } \Omega, \quad \phi \in C^2(\Omega\setminus\{a_1, \ldots, a_n\}).$$
or equivalently, if there exists a solution for the problem
\be\label{supersol}
\left\{
\begin{array}{cc}
L_\mu \phi(x)= 0, & \textrm{ a.e.  } x\in  \Omega\\
\phi(x)>0,  & \forall x\in \Omega.\\
\phi \in C^2(\Omega\setminus\{a_1, \ldots, a_n\}). &\\
\end{array}\right.
\ee
In view of \eqref{gen}, solutions $\phi's$ as in \eqref{supersol}, that we will call ``weights", allow in particular to get the Hardy inequality
\be\label{hardpot}
\into |\n u|^2 \dx \geq \mu \into V u^2 \dx, \quad \forall u\in C_{0}^{1}(\Omega\setminus \{a_1, \ldots, a_n\}).
 \ee
 The extension of \eqref{hardpot} to the appropriate functional space $\semi$    is possible since $C_0^1(\Omega\setminus \{a_1, \ldots, a_n\})$ is dense in $\semi$ (see for instance \cite{hiddenenergy} for similar density results in the case of a single pole).
Here, a precise aim is to construct weights $\phi$ satisfying \eqref{supersol} in order to maximize $\mu$ in \eqref{hardpot}. We will analyze both situations in which the poles are located either in the interior or on the boundary of $\Omega$.

As in  \cite{cristi1} we first consider $\phi$ in \eqref{gen} of the form
\be\label{produ}
\phi=(\phi_1\phi_2\ldots \phi_n)^{\frac{1}{n}},
\ee
where the weights $\phi_i\in \mathcal{D}'(\Omega)$, $i\in \{1, \ldots, n\}$,  are supposed to satisfy some admissible conditions adapted to each singularity $a_i$.
Then $\phi$ verifies
\be\label{genhard1}
-\frac{\D \phi}{\phi}= \frac{1}{n^2}\sum_{1\leq i<j \leq n}\left|\frac{\n \phi_i}{\phi_i}-\frac{\n \phi_j}{\phi_j}\right|^2-\frac{1}{n}\sum_{i=1}^{n}\frac{\Delta \phi_i}{\phi_i}.
\ee
Assume now that each weight $\phi_i$, with $i\in \{1, \ldots, n\}$,  fulfills the conditions
\be\label{impcond2}
\left\{\begin{array}{ll}
 \phi_i\in C^{2}(\Omega \setminus
\{a_1, a_2, \ldots a_n\}),  & \\
  \phi_i(x)>0 \textrm{ in } \Omega, &\\
 -\D \phi_i(x)= 0, \q \forall x\in \Omega\setminus\{a_1, \ldots, a_n\}.   &
\end{array}\right.
\ee
Then, according to \eqref{produ}-\eqref{impcond2}  and \eqref{gen} we conclude that any  $ u \in \semi$ verifies
\begin{equation}\label{buni}
\into |\n u|^2 \dx - \frac{1}{n^2}\sum_{1\leq i<j \leq n} \into \left|\frac{\n \phi_i}{\phi_i}-\frac{\n \phi_j}{\phi_j}\right|^2 u^2 \dx
=\into \left |\n \left(u \prod_{i=1}^{n} \phi_i^{-1/n}\right)\right|^2 \prod_{i=1}^{n}\phi_i^{2/n} \dx.
\end{equation}
Identity \eqref{buni} is very important in our analysis since it applies to the proofs of  Theorem \ref{prop1} and Theorem \ref{t1} in both contexts of  interior and boundary singularities, by choosing the weights $\phi_i$ in an efficient way, as emphasized in the next sections.
 \section{Interior singularities: proofs of Theorems \ref{prop1}-\ref{moregen}}\label{int.sing}
This section is devoted to prove both Theorem \ref{prop1} and Theorem \ref{moregen}. For the proof of Theorem \ref{prop1} we will make use of \eqref{buni} as follows.

\subsection{Proof of Theorem \ref{prop1}}\label{prooft1}  We split the proof in several steps.

\noindent{\bf Step I: election of the weights $\phi_i$ in \eqref{buni}.}
In \cite{cristi1}, in the context of interior singularities, when $\Omega=\rr^N$  we chose
$$\phi_i=|x-a_i|^{-(N-2)},$$
   (this weight corresponds to the fundamental solution of the Laplacian at the point $a_i$) which  in particular  satisfies \eqref{impcond2} for all $i\in \{1, \ldots, n\}$ and the following identity holds true:
\be\label{ident}
\sum_{1\leq i<j \leq n}\left|\frac{\n \phi_i}{\phi_i}-\frac{\n \phi_j}{\phi_j}\right|^2=(N-2)^2 V.
\ee
Applying  \eqref{ident} and \eqref{buni} we obtain
\begin{multline}\label{optint}
\into |\n u|^2 \dx - \frac{(N-2)^2}{n^2} \into V u^2 \dx\\
=\into \left|\n \left( u \prod_{i=1}^{n} |x-a_i|^{(N-2)/n}\right)\right|^2  \prod_{i=1}^{n} |x-a_i|^{-2(N-2)/n} \dx  >0, \quad \forall u\in \semi.
\end{multline}

Next we apply \eqref{optint} to show \eqref{equality.constant}. We treat separately the bipolar case $n=2$ and the multipolar case $n\geq 3$.

\noindent{\bf Step II: Proof of \eqref{equality.constant} in the case $n=2$.}

In view of \eqref{wholespace} we easily obtain that  $\mu^\star(\Omega)\geq (N-2)^2/4$. In order to prove the reverse inequality it suffices to build a minimizing sequence $\{u_\eps\}_{\eps>0}\subset \semi$ satisfying
\be\label{limit.interior}
\frac{\into |\n u_\eps|^2\dx }{\into V u_\eps^2 \dx } \longrightarrow \frac{(N-2)^2}{4}, \quad \textrm{ as }\eps\rightarrow 0.
\ee
 Observe that the minimizing sequence proposed in \cite{cristi1} for the proof of \eqref{optimal.global} is not suitable here because its support spreads everywhere in $\rr^N$ when $\eps$ tends to zero.   On the contrary, in the following we manually construct a minimizing sequence $\{u_\eps\}_{\eps>0}$ as in  \eqref{limit.interior} with the support concentrated near the singular poles when $\eps$ tends to zero.

 So,   let $\eps>0$ small enough such that $\eps<\min\{1,  M^2/4\}$. Then, it follows that  $B_{\eps^{1/2}}(a_1)\cap B_{\eps^{1/2}}(a_2)=\emptyset$ and  $\cup _{i=1}^{2} B_{\eps^{1/2}}(a_i)\subset \Omega$.

 Let us also define the continuous functions $\theta_\eps$ given by
\be\label{theta.eps}
\theta_\eps(x):=\left\{
\begin{array}{ll}
  0, & x\in B_{\eps^2}(a_i), \quad i\in \{1,2\}, \\[3pt]
  \frac{\log\left(|x-a_i|/\eps^2\right)}{\log 1/\eps},  &  x \in B_{\eps}(a_i)\setminus B_{\eps^2}(a_i), \quad i\in \{1,2\}, \\[3pt]
  \frac{2\log\left(\eps^{1/2}/|x-a_i|\right)}{\log 1/\eps}, & x\in B_{\eps^{1/2}}(a_i)\setminus B_{\eps}(a_i), \quad i\in \{1, 2\}, \\[3pt]
  0,  & \textrm{ otherwise}.
\end{array}\right.
\ee
Observe that $\theta_\eps\in \hoi$ is compactly supported in $\Omega$  with the support localized in annulus regions surrounding the singularities. Motivated by  \eqref{optint} we consider the sequence  $\{u_\eps\}_{\eps>0}\subset \hoi(\subset \semi)$ defined by
\be\label{seq.int}
u_\eps(x):=\prod_{i=1}^{2} |x-a_i|^{-(N-2)/2}\theta_\eps(x).
\ee
For $u_\eps$ in \eqref{seq.int} identity \eqref{optint}  becomes \be\label{ident.int}
\into |\n u_\eps|^2\dx- \frac{(N-2)^2}{4}\into V u_\eps^2 \dx  =\into |\n \theta_\eps|^2 \prod_{i=1}^{2} |x-a_i|^{-(N-2)}\dx.
\ee
  Detailing the computations for $\n \theta_\eps$ we split the right hand side of \eqref{ident.int} in $I_1+I_2$, where
$$I_1:=\left(\log \frac{1}{\eps}\right)^{-2}  \sum_{i=1}^{2}\int_{B_\eps(a_i)\setminus B_{\eps^2}(a_i)} |x-a_i|^{-2} \prod_{j=1}^{2} |x-a_j|^{-(N-2)}\dx,$$
$$ I_2 := 4\left(\log \frac{1}{\eps}\right)^{-2}  \sum_{i=1}^{2}\int_{B_{\eps^{1/2}}(a_i)\setminus B_{\eps}(a_i)} |x-a_i|^{-2} \prod_{j=1}^{2} |x-a_j|^{-(N-2)}\dx. $$
Besides, for $\eps<\min\{1, M^2/4\}$ we have
\begin{equation}\label{tavica}
\frac{M}{2} < |x-a_j| <\frac 3 2 d, \quad \forall x\in B_{\eps^{1/2}}(a_i), \  \forall  i, j \in \{1, 2\}, \  i\neq j.
 \end{equation}
 Then we successively obtain
\begin{align*}
I_1 & \leq \left(\log \frac{1}{\eps}\right)^{-2} \left(\frac{M}{2}\right)^{2-N} \sum_{i=1}^{2}\int_{B_\eps(a_i)\setminus B_{\eps^2}(a_i)} |x-a_i|^{-N} \dx  \nonumber\\
&= 2 \left(\log \frac{1}{\eps}\right)^{-2} \left(\frac{M}{2}\right)^{2-N}  \omega_N \int_{\eps^2}^{\eps} \frac{1}{r} \dr =  2 \left(\frac{M}{2}\right)^{2-N} \omega_N \left(\log \frac{1}{\eps}\right)^{-1},
\end{align*}
where $\omega_N$ is the $(N-1)$-Hausdorff measure of the unit sphere
$S^{N-1}$ in $\rr^N$. From above we get
\be\label{upper.bound1}
I_1\leq C_1 \left(\log \frac{1}{\eps}\right)^{-1},
\ee
for some positive constant $C_1$ independent of $\eps$.

Analogously,  it  follows that
\begin{align}\label{upper.bound2}
I_2 & \leq C_2 \left(\log \frac{1}{\eps}\right)^{-1},
\end{align}
for some $C_2>0$ independent of $\eps$.  Particularly, the upper bounds \eqref{upper.bound1}-\eqref{upper.bound2} yield to
\be\label{unif}
\into |\n \theta_\eps|^2 \prod_{i=1}^{2} |x-a_i|^{-(N-2)}\dx \leq  C_3 \left(\log \frac{1}{\eps}\right)^{-1},
\ee
where $C_3=C_1+C_2$.
Next we will determine lower bounds for the potential term
$$ I_V:=\into V u_\eps^2 \dx.$$
First we obviously have $I_V>I_{V,1}$ where
$$I_{V,1} :=\int_{\cup_{i=1}^{2} B_{\eps}(a_i)\setminus B_{\eps^2}(a_i)} V u_\eps^2 \dx. $$
Applying the upper bound in \eqref{tavica}, explicit computations lead to
\begin{align*}
I_{V,1} &= d^2  \left(\log \frac{1}{\eps}\right)^{-2} \sum_{i=1}^{2}\int_{B_{\eps}(a_i)\setminus B_{\eps^2}(a_i)} \left(\log \frac{|x-a_i|}{\eps^2}\right)^{2} \prod_{j=1}^{2} |x-a_j|^{-N}\dx \nonumber\\
& \geq d^2 \left(\log \frac{1}{\eps}\right)^{-2} \left(
\frac{3}{2}d \right)^{-N} \sum_{i=1}^{2}\int_{B_{\eps}(a_i)\setminus B_{\eps^2}(a_i)} \left(\log \frac{|x-a_i|}{\eps^2}\right)^{2} |x-a_i|^{-N}\dx \nonumber\\
&= 2^{N+1} 3^{-N} d^{2-N} \left(\log \frac{1}{\eps}\right)^{-2} \omega_N \int_{\eps^2}^{\eps} \left(\log \frac{r}{\eps^2}\right)^{2} \frac{1}{r} \dr \nonumber\\
&= \left(\frac{2}{3}\right)^{N+1} d^{2-N} \omega_N \left(\log \frac{1}{\eps}\right).
\end{align*}
 Therefore, there exists $C_4>0$ such that
  \begin{equation}\label{lower.bound1}
  I_{V} \geq C_4 \log \frac{1}{\eps}.
  \end{equation}
 Combining \eqref{unif}-\eqref{lower.bound1} with \eqref{ident.int} we finally obtain
  $$\frac{(N-2)^2}{4} \leq \frac{\into |\n u_\eps|^2 \dx}{\into V u_\eps^2 \dx }\leq  \frac{(N-2)^2}{4}+\frac{C_3}{C_4} \left(\log \frac{1}{\eps}\right)^{-2}.$$
  Letting $\eps\to 0$, we complete the proof of  \eqref{equality.constant} in the case $n=2$. Moreover, the minimizing sequence $u_\eps$ in \eqref{seq.int} also applies to prove \eqref{optimal.global} in unbounded domains $\Omega$  since the support of $u_\eps$ concentrates near the singular poles as $\eps$ tends to zero.

\noindent{\bf Step III. Proof of \eqref{equality.constant} in the case $n\geq 3$.}

 In order to prove the strict lower bound in \eqref{equality.constant} we employ identity \eqref{optint} and the following  weighted Hardy inequality

\begin{lema}\label{lema2} Let  $\Omega\subset \rr^N$, $N\geq 3$, be an open bounded domain such that $a_1, \ldots, a_n\in \Omega$. Then, for any $\varphi\in C_{0}^{1}(\Omega \setminus\{a_1, \ldots, a_n\})$ and any $i \in \{1, \ldots, n\}$  we have
\be\label{niceineq}
\into |\n \varphi|^2 |x-a_i|^{-2(N-2)/n} \dx \ge \frac{(N-2)^2\left(1-\frac{2}{n}\right)^2}{4} \into \varphi^2 |x-a_i|^{-2(N-2)/n-2}  \dx.
\ee
\end{lema}
\begin{proof}
We start to estimate the right hand side in \eqref{niceineq}. Writing the potential term in the divergence form and using integration by parts we successively get
\begin{align}\label{for1}
 \into \varphi^2 |x-a_i|^{-2(N-2)/n-2}  \dx &= \frac{1}{(N-2)(1-2/n)}\into \textrm{div}\left(\frac{x-a_i}{|x-a_i|^{2(N-2)/n+2}} \right) \varphi^2 \dx\nonumber\\
 &=-\frac{2}{(N-2)(1-2/n)} \int \varphi \nabla \varphi \cdot \frac{x-a_i}{|x-a_i|^{2(N-2)/n+2}} \dx.
\end{align}
Applying the Cauchy-Schwarz inequality  in \eqref{for1} we obtain
\begin{align*}
\into \varphi^2 |x-a_i|^{-2(N-2)/n-2}  \dx &\leq \frac{2}{(N-2)(1-2/n)} \left(\into |\nabla \varphi |^2 |x-a_i|^{-2(N-2)/n} \dx \right)^{1/2}\times\\
&\times  \left( \into \varphi^2 |x-a_i|^{-2(N-2)/n-2} \dx  \right)^{1/2}
  \end{align*}
  which is equivalent to \eqref{niceineq} after taking the squares. The proof of Lemma \ref{lema2} is now finished.
\end{proof}
Now we turn back to Step III and we set
  $$\kappa_1:=\inf_{x\in \Omega} \frac{\prod_{i=1}^{n} |x-a_i|^{-2(N-2)/n}}{\sum_{i=1}^{n} |x-a_i|^{-2(N-2)/n}}, \quad \kappa_2:=\inf_{x\in \Omega} \frac{\sum_{i=1}^{n} |x-a_i|^{-2(N-2)/n-2}}{V \prod_{i=1}^{n} |x-a_i|^{-2(N-2)/n}}.$$
  Observe that, $0< \kappa_1, \kappa_2 < \infty$. Indeed, when we focus on the definition of $\kappa_1$ we notice that both terms involved in the quotient behave like $|x-a_i|^{-2(N-2)/n}$ (up to a multiplicative constant) at each singular pole $a_i$, $i\in \{1, \ldots, n\}$. Whereas, far from the singularities both the numerator and denominator are bounded and strictly positive. It happens  similarly with the terms in the quotient of $\kappa_2$.

Now let $u\in C_0^1(\Omega\setminus\{a_1, \ldots, a_n\})$ and set $\varphi:= u \prod_{i=1}^{n} |x-a_i|^{(N-2)/n}$. Next, in view of \eqref{optint}, Lemma \ref{lema2} applied to $\varphi$  and the definitions of $\kappa_1$ and $\kappa_2$ we successively get
  \begin{align}\label{import1}
  \into |\n u|^2 \dx -\frac{(N-2)^2}{n^2}\into V u^2 \dx & = \into |\n \varphi |^2 \prod_{i=1}^{n} |x-a_i|^{-2(N-2)/n} \dx \nonumber\\
  & \geq \kappa_1 \sum_{i=1}^{n} \into |\n \varphi|^2 |x-a_i|^{-2(N-2)/n} \dx\nonumber\\
  & \geq \kappa_1 \frac{(N-2)^2\left(1-\frac{2}{n}\right)^2}{4}\sum_{i=1}^{n} \into \varphi^2 |x-a_i|^{-2(N-2)/n-2} \dx\nonumber\\
  &\geq \kappa_1 \kappa_2 \frac{(N-2)^2\left(1-\frac{2}{n}\right)^2}{4} \into V \varphi^2 \prod_{i=1}^{n} |x-a_i|^{-2(N-2)/n} \dx \nonumber\\
  &= \kappa_1 \kappa_2 \frac{(N-2)^2\left(1-\frac{2}{n}\right)^2}{4} \into V u^2 \dx.
  \end{align}
Therefore we obtain,
 $$ \into |\n u|^2 \dx \geq \left(\frac{(N-2)^2}{n^2}+ \kappa_1 \kappa_2 \frac{(N-2)^2\left(1-\frac{2}{n}\right)^2}{4} \right)\into V u^2 \dx,$$
 which by density can be extended to any $u\in \semi$. In consequence, we conclude that
 \begin{align*}
 \mu^\star(\Omega)&\geq \frac{(N-2)^2}{n^2}+ \kappa_1 \kappa_2 \frac{(N-2)^2\left(1-\frac{2}{n}\right)^2}{4}\\
 & > \frac{(N-2)^2}{n^2},
 \end{align*}
 since $n\geq 3$. This yields to the proof of the strict lower bound in \eqref{equality.constant}.
\endproof

The proof of the upper bound in \eqref{equality.constant} is based on the following lemma which also applies to the proof of Theorem \ref{moregen}.

\begin{lema}[local results for interior singularities]\label{lema1.int}
Assume $\Omega\subset \rr^N$, $N\geq 3$, is an open domain such that $a_1, \ldots, a_n\in \Omega$, $n\geq 2$.
For every $\eps>0$ small enough,  there exists $U_{\eps}$ a neighborhood  of $\cup_{i=1}^{n} \{a_i\}$ in $\Omega$ such that:
\begin{enumerate}[(i).]
\item\label{1.int} For any $u\in C_{0}^{\infty}(U_\eps)$ it follows that
\be\label{local.hardy.int}
\into |\n u|^2\dx> \left(\frac{(N-2)^2}{4(n-1)}-\eps \right)\into V u^2\dx.\ee
\item\label{2.int} There exists $u_\eps\in C_{0}^{\infty}(U_\eps)$ satisfying
\be\label{local.revers.int}
\into |\n u_\eps|^2 \dx < \left(\frac{(N-2)^2}{4(n-1)}+\eps\right)\into  V u_\eps^2 \dx.
\ee
\end{enumerate}
\end{lema}
\begin{proof}
 Roughly speaking,  the proof follows combining the local behavior of the potential $V$ near each pole $a_i$ and the  Hardy inequality \eqref{classical_Hardy_inequality} with the singularity shifted at each $a_i$. Indeed, we first remark that $V$ can also be expressed as
\be
V(x)=\frac{1}{|x-a_i|^2}\left[\sum_{j=1, j\neq i}^{n} \frac{|a_i-a_j|^2}{|x-a_j|^2}+ O(|x-a_i|^2) \right], \textrm{ as } x\to a_i,  \nonumber
\ee
for any $i\in \{1, \ldots, n\}$. In consequence,
\be\label{asympt.loc.int}
\lim_{x\to a_i}V(x)|x-a_i|^2=(n-1).
\ee
  Let $r>0$ small enough such that
   $$\cup_{i=1}^{n} B_r(a_i)\subset \Omega \quad  \textrm{ and } \quad  B_r(a_i)\cap B_r(a_j)=\emptyset,  \  \forall i\neq j.  $$
Then, in view of \eqref{classical_Hardy_inequality} we have    \be\label{loc.bound.fall.int}
\int_{B_r(a_i)} |\n u|^2\dx \geq \frac{(N-2)^2}{4}\int_{B_r(a_i)} \frac{u^2}{|x-a_i|^2}\dx, \quad \forall u\in C_{0}^{\infty}(B_{r}(a_i)), \quad \forall i.
\ee

Now, let be $\eps>0$ and set $\delta_\eps:=4(n-1)^2\eps/\left((N-2)^2-4(n-1)\eps\right)$ which is positive  for $\eps$ small enough.  In view of \eqref{asympt.loc.int}, there exists $r_\eps>0$  such that
\be\label{limit.int}
n-1-\delta_\eps <  V(x)|x-a_i|^2 < n-1+\delta_\eps, \quad \forall x\in B_{r_\eps}(a_i), \quad \forall i.
\ee
Actually,  we may assume  $r_\eps$ small enough such that $r_\eps <r$ and  consider then $U_\eps:=\cup_{i=1}^{n}  B_{r_\eps}(a_i)$.  Let be $u\in C_{0}^{\infty}(U_\eps)$ and denote  $u_i:=u_{|B_{r_\eps}(a_i)}$. Since $u$ is supported in $n$ disjoint balls that shrink around the singular poles, applying  \eqref{loc.bound.fall.int} and \eqref{limit.int} for each $u_i$ we successively have
\begin{align}\label{ineq2}
\into |\n u|^2 \dx &= \sum_{i=1}^{n} \into |\n u_i|^2 \dx \nonumber\\
& \geq \frac{(N-2)^2}{4}\frac{1}{n-1+\delta_\eps} \sum_{i=1}^{n} \into V u_i^2 \dx \nonumber\\
&  =\frac{(N-2)^2}{4}\frac{1}{n-1+\delta_\eps} \into V u^2 \dx\nonumber\\
&= \left(\frac{(N-2)^2}{4(n-1)}-\eps\right) \into V u^2 \dx,
\end{align}
and the proof of \eqref{local.hardy.int} is finally obtained.

For the proof of \eqref{local.revers.int} we mainly use the optimality of the constant $(N-2)^2/4$ in \eqref{loc.bound.fall.int} as follows.

Let $\eps>0$  and fix $0< \delta_\eps^{'}:=(4n^2-10n+6)\eps/\left((N-2)^2+4(n-1)\eps\right)$. Observe that $\delta_\eps^{'}< \delta_\eps$ and therefore \eqref{limit.int} is also valid for $\delta_{\eps}^{'}$ instead of $\delta_{\eps}$ for some $r_{\eps}^{'}$ instead of $r_\eps$, with $r_{\eps}^{'}< r_\eps$.

Due to the optimality of $(N-2)^2/4$ in \eqref{loc.bound.fall.int} for $r=r_{\eps}^{'}$, there exists $u_{i, \eps}\in C_{0}^{\infty}(B_{r_{\eps}^{'}}(a_i))$ such that
\be\label{ineq.invers.int}
\int_{B_{r_{\eps}^{'}}(a_i)} |\n u_{i,\eps}|^2 \dx \leq \left(\frac{(N-2)^2}{4}+\frac \eps 2 \right) \int_{B_{r_{\eps}^{'}}(a_i)} \frac{u_{i, \eps}^2}{|x-a_i|^2} \dx.
\ee
Then we consider $u_\eps:=\sum_{i=1}^{n} u_{i, \eps} \chi_{B_{r_{\eps}^{'}}(a_i)}$ which satisfies $u_\eps\in C_{0}^{\infty}(U_\eps)$.

Therefore, combining \eqref{ineq.invers.int} and \eqref{limit.int} we get \begin{align}
\into |\n u_\eps|^2 \dx &= \sum_{i=1}^{n}\into |\n u_{i, \eps}|^2 \dx\nonumber\\
& \leq \left(\frac{(N-2)^2}{4}+\frac \eps 2 \right)\frac{1}{n-1-\delta_\eps^{'}}\sum_{i=1}^{n} \into V u_{i, \eps}^{2}\dx\nonumber\\
&= \left(\frac{(N-2)^2}{4(n-1)}+\eps\right)\sum_{i=1}^{n} \into V u_{i, \eps}^{2}\dx\nonumber\\
&= \left(\frac{(N-2)^2}{4(n-1)}+\eps\right)\into V u_\eps^2 \dx.
\end{align}
The proof of Lemma \ref{lema1.int} is  finished.
\end{proof}
The upper bound in \eqref{equality.constant} is a direct consequence of \eqref{local.revers.int} in  Lemma \ref{lema1.int}. Finally, this completes the proof of Theorem \ref{prop1}.

\subsection{Proof of Theorem \ref{moregen}}\label{sec3}

In the following we will apply again Lemma \ref{lema1.int}.

\subsubsection*{ Proof of item \eqref{item1}.}
Let us consider $\eps>0$ such that $\mu<(N-2)^2/4-\eps$, and let $U_{r_\eps}:=\cup_{i=1}^{n}B_{r_\eps}(a_i)$,  $r_\eps>0$,  be a neighborhood of $\cup_{i=1}^{n}\{a_i\}$ in $\Omega$ such that $B_{r_\eps}(a_i)\cap B_{r_\eps}(a_j)=\emptyset $ for any $i\neq j$, $i, j\in \{1, \ldots, n\}$, satisfying \eqref{local.hardy.int}. Moreover, let $\xi_\eps$ be a $C^2(\overline{\Omega})$ cut-off function, so that $0\leq \xi_\eps\leq 1$, $\xi_\eps\equiv 1$ in  $\cup_{i=1}^{n}B_{r_\eps/2}(a_i)$ and the support of $\xi_\eps$ is contained in $U_{r_\eps}$. By integration by parts, for any $u\in C_{0}^{\infty}(\Omega)$ we have
\begin{align}\label{bb1}
\into |\n (\xi_\eps u)|^2 \dx &=\into \xi_\eps^2 |\n u|^2 \dx -\into (\xi_\eps \Delta \xi_\eps ) u^2 \dx \nonumber\\
&\leq \into |\n u|^2 \dx + C_\eps \into u^2 \dx,
\end{align}
for some constant $C_\eps>0$ depending on $\eps$.
Then, from  \eqref{local.hardy.int} in Lemma \ref{lema1.int} we get
\begin{align}\label{bb2}
\into |\n (\xi_\eps u)|^2 \dx& \geq \mu \into V \xi_\eps^2 u^2 \dx \nonumber\\
& \geq \mu \into V u^2 \dx - D_\eps \int_{\Omega\setminus U_{r_\eps/2}} u^2  \dx,
\end{align}
for some constant $D_\eps>0$, since $V$ is bounded in $\Omega\setminus U_{r_\eps/2}$.
Therefore, from \eqref{bb1}-\eqref{bb2} we conclude
$$\into |\n u|^2 \dx + (C_\eps-D_\eps)\into u^2 \dx \geq \mu \into V u^2 \dx, \quad \forall u \in C_{0}^{\infty}(\Omega), $$
and \eqref{good1} is proven.

\subsubsection*{ Proof of item \eqref{item2}.}
Let $\lambda\in \rr$ and $\mu> (N-2)^2/(4n-4)$ be fixed. Let us also consider $\mu'$ and $\eps>0$ such that  $\mu> \mu' > (N-2)^2/(4n-4)$ and  $(N-2)^2/(4n-4)+\eps < \mu'$.

 In view of \eqref{local.revers.int} in  Lemma \ref{lema1.int} there exists a neighborhood $U_{\eps}:=\cup_{i=1}^{n}B_{r_\eps}(a_i)\cap \Omega$ of $\cup_{i=1}^{n}\{a_i\}$ in $\Omega$ and  there exists $u_\eps\in C_0^\infty(U_{\eps})$ satisfying
 \begin{align}\label{up.hardy}
 \int_{U_{\eps}} |\n u_\eps|^2 \dx &< \left(\frac{N^2}{4(n-1)}+\eps\right)\int_{U_{\eps}}  V u_\eps^2\dx\nonumber\\
& < \mu' \int_{U_{\eps}} V u_\eps^2 \dx.
\end{align}
It follows that
$$(\mu-\mu')\int_{U_{\eps}} V u_\eps^2 \dx +\int_{U_{\eps}} |\n u_\eps|^2\dx < \mu \int_{U_{\eps}} V u_\eps^2 \dx.$$
Finally, it is enough to consider $r_\eps$ small enough such that
$$\frac{\lambda}{\mu-\mu'}\leq \inf_{U_{\eps}} V,$$
to conclude with
$$\lambda \int_{U_{\eps}}u_\eps^2 \dx  +\int_{U_{\eps}} |\n u_\eps|^2\dx < \mu \int_{U_{\eps}} V u_\eps^2 \dx. $$

\subsubsection*{ Proof of item \eqref{item3}.}\label{sec4}
Let $\{u_n\}_{n\in \nn}\subset \hoi$ be a normalized minimizing sequence such that
\be\label{con1}
\into |\n u_n|^2 \dx=\mu^\star(\Omega),
\ee
\be\label{con2}
\int Vu_n^2 \dx=1+ o(1), \textrm{ as } n \to \infty.
\ee
Then there exists $u\in \hoi$ s.t.
\be\label{con3}
u_n \rightharpoonup u \textrm{ in } \hoi.
\ee
 As a consequence of Theorem \ref{prop1} the embedding  $\hoi\hookrightarrow L^2(V; \dx)$ is continuous  and $\hoi$ is compact embedded in $L^2(\Omega)$. Then we get
\be\label{con4}
u_n \to u \textrm{ in } L^2(\Omega), \textrm{ and } u_n \rightharpoonup u \textrm{ in } L^2(V; \dx).
\ee
Put $e_n:=u_n-u$. We apply \eqref{con1} and \eqref{con2} to obtain
\be\label{con5}
\into |\n u|^2 \dx+ \into |\n e_n|^2 \dx= \mu^\star(\Omega)+o(1),
\ee
\be\label{con6}
\into  Vu^2 \dx + \into V e_n^2 \dx= 1+ o(1), \textrm{ as } n \to \infty.
\ee
Now, let us fix $\delta>0$ such that $\mu^\star(\Omega)+\delta < (N-2)^2/(4n-4)$ and choose $\mu \in (\mu^\star(\Omega)+\delta , (N-2)^2/(4n-4))$.

Then, according to Theorem \ref{moregen} it follows that
\be\label{con7}
\into |\n e_n|^2 \dx \geq \mu\into V e_n^2 \dx + o(1),
\ee
since $e_n\to 0$ in $L^2(\Omega)$ due to \eqref{con4}.

Therefore, applying \eqref{con5}-\eqref{con7} we successively have
\begin{align}\label{suc}
\mu^\star(\Omega)\into V u^2 \dx & \leq \into |\n u|^2 \dx= \mu^\star(\Omega)-\into |\n e_n|^2 \dx+o(1)\nonumber\\
& \leq \mu^\star(\Omega)-(\mu^\star(\Omega)+\delta)\into V e_n^2 \dx +o(1)\nonumber\\
&\leq \mu^\star(\Omega)\left(\into V u^2 \dx + \into V e_n^2 \dx \right)-(\mu^\star(\Omega)+\delta)\into V e_n^2 \dx + o(1)\nonumber\\
& = \mu^\star(\Omega)\into V u^2 \dx- \delta \into Ve_n^2 \dx +o(1).
\end{align}
Since $\delta>0$, we must have $\into V e_n^2 \dx \to 0$, as $n\to \infty$. Coming back to \eqref{con5} and \eqref{con6} we obtain $\into V u^2 \dx =1$ (so $u\neq 0$) and $\into |\n u|^2 \dx \leq \mu^\star(\Omega)$. Taking into account the definition of $\mu^\star(\Omega)$ we get that $u$ satisfies
$$\mu^\star(\Omega)\into V u^2 \dx = \into |\n u|^2 \dx,$$
so the constant is attained by a non-trivial function $u\in \hoi$.

\subsubsection*{ Proof of item \eqref{item4}.} Since $n=2$ we deduce from Theorem \ref{prop1} that $\mu^\star(\Omega)=(N-2)^2/4$.
 Assume that $\mu^\star(\Omega)=(N-2)^2/4$ is attained by a function $u\in \semi$. Then, $u$ satisfies identity \eqref{optint} in which the left hand side vanishes. Therefore, the right hand side also vanishes. This implies $u=C\prod_{i=1}^{2}|x-a_i|^{-(N-2)/2}$ for some nontrivial constant $C$. But such $u$ does not belong to $ \semi$ because neither $u$ vanishes on the boundary of $\Omega$, nor  $\n u\not \in L_{loc}^2(\Omega)$. Therefore, $\mu^\star(\Omega)$ is not attained when $n=2$.

 The proof of Theorem \ref{moregen} is now complete.

\section{Boundary singularities and proofs of  main results}\label{bound.sing}
In this section we analyze the situation in which all the singular poles are located at the boundary. More specific, we prove both Theorem \ref{t1} and Theorem \ref{te2}. For the proof of Theorem \ref{t1} we also start from identity \eqref{buni}  in which we have to consider other weights $\phi_i$  than  those chosen in the case of interior singularities.

\subsection{Preliminary weights}
 Of course, the choice of the weights $\phi_i=|x-a_i|^{-(N-2)}$  in Section \ref{prooft1}  is still admissible  in this new case. Despite of that, these weights do not provide an optimal Hardy constant $\mu^\star(\Omega)$ for boundary singularities, merely  ensuring $\mu^\star(\Omega)\geq (N-2)^2/n^2$.

 In order to get optimal results we design new weights $\phi_i$ in \eqref{buni} with the profile adapted to both the whole geometry of $\Omega$ and the corresponding singularity $a_i$, $i\in \{1, \ldots, n\}$.

In view of that, in order to remove the singularities $a_i$ we first introduce the weights
\be\label{view}
\phi_i(x)=f(x)|x-a_i|^{-N},\quad  x\in \Omega, \ i\in \{1, \ldots, n\},
\ee
where $f$ is a function independent of $a_i$ whose restrictions will be precise later. Then we obtain
\be\label{grad}
\frac{\n \phi_i}{\phi_i}= \frac{\n f}{f}-\frac{N(x-a_i)}{|x-a_i|^2}, \quad \quad
\left| \frac{\n \phi_i}{\phi_i}-\frac{\n \phi_j}{\phi_j} \right|^2= N^2 \frac{|a_i-a_j|^2}{|x-a_i|^2|x-a_j|^2}
\ee
and
\be\label{grad1}
\Delta \phi_i=\left(\D f |x-a_i|^2 -2 N \n f \cdot (x-a_i) +2 N f\right) |x-a_i|^{-N-2}.
\ee
Let us consider $f$ in \eqref{view} with the following properties:
 \be\label{impcond4}
\left\{\begin{array}{ll}
 f\in C^{2}(\Omega),  & \\
  f(x)>0, & \forall x\in \Omega, \\
S(f, a_i):=-\D f |x-a_i|^2 +2N \n f \cdot (x-a_i) -2N f = 0, & \q \forall x\in \Omega, \\
\end{array}\right.
\ee
for any $i\in\{1, \ldots, n\}$. Then, the general assumptions \eqref{impcond2} are satisfied for each $\phi_i$ in \eqref{view}. Therefore, applying \eqref{buni} with the weights in \eqref{view}  it follows
\begin{align}\label{impcon}
\into |\n u|^2 \dx &= \frac{N^2}{n^2} \into V u^2 \dx \nonumber\\
& + \into \left|\n\left[ u \left( f \prod_{i=1}^{n} |x-a_i|^{-\frac{N}{n}} \right)^{-1}\right] \right|^2 f^2 \prod_{i=1}^{n} |x-a_i|^{-\frac{2N}{n}} \dx
\end{align}
 for any $u\in \semi$.

Thus, if $f$ fulfills \eqref{impcond4} from \eqref{impcon} we get the improved Hardy inequality
\be\label{hardin}
\into |\n u|^2 \dx \geq \frac{N^2}{n^2}\into  V u^2 \dx, \quad \forall u\in \semi,
\ee
which is verified in several particular cases as shown in Theorem \ref{t1}.

\subsection{Proof of Theorem \ref{t1}}\label{sec1}
Next we apply identity \eqref{impcon} to prove Theorem \ref{t1} by constructing suitable weights $f$ as in \eqref{impcond4}. We analyze separately the cases when $\Omega$ is either a ball, the exterior of a ball or a half-space. For the simplicity of redaction when there is no risk of confusions we will maintain the notation $\Omega$ instead of $B_r(x_0)$, $B_r^c(x_0)$ or $\rr_+^N$.\\

\subsubsection{The case of a ball}\label{secball}

 Let $\Omega=B_r(x_0)$ for some   $r>0$ and $x_0\in \rr^N$, and let the singular poles $a_1, \ldots, a_n$ to be placed on the boundary of $\Omega$.  The proof follows several steps.

\noindent{\bf Step I: Election of $f$ in \eqref{impcond4}.} We show that $\mu^\star(B_r(x_0))\geq N^2/n^2$.
 For that to be true, it suffices  to build a function $f$ satisfying  \eqref{impcond4} in  $B_r(x_0)$. In view of that, let us introduce
\begin{equation}\label{natelec}
f(x):=r^2-|x-x_0|^2.
\end{equation}
 The first two conditions in \eqref{impcond4}  are trivially satisfied.
 The third condition in \eqref{impcond4} is also true: expanding the square we get for any $x\in B_r(x_0)$ that
\begin{align}\label{cr}
S(f;a_i)&= 2N |x-a_i|^2 -4N (x-x_0)\cdot (x-a_i) -2N (r^2- |x-x_0|^2)\nonumber\\
 &= -2Nr^2 +2N |(x-a_i)-(x-x_0)|^2 \nonumber\\
 &= -2N r^2 +2N |a_i-x_0|^2\nonumber\\
 &= 2N (|a_i-x_0|^2-r^2)\nonumber\\
 &=0,
\end{align}In the last identity we applied the fact that the singularities $a_i$ are located on the boundary of $B_r(x_0)$.

\noindent{\bf Step II: Optimality  and attainability of $N^2/n^2$ in \eqref{hardin} in the case $n\geq 3$.}

 Since $\Omega$ is bounded then \eqref{impcon} applies to functions in $\hoi$:
\begin{align}\label{impid}
\into |\n u|^2 \dx &- \frac{N^2}{n^2} \into V u^2 \dx \nonumber\\
 &=\into \left|\n\left[ u \left( f \prod_{i=1}^{n} |x-a_i|^{-\frac{N}{n}} \right)^{-1}\right] \right|^2 f^2 \prod_{i=1}^{n} |x-a_i|^{-\frac{2N}{n}} \dx, \quad \forall u\in \hoi,
\end{align}
where $f$ is as in \eqref{natelec}.

Next we consider the function $u$ defined by
\be\label{bob}
u:=(r^2-|x-x_0|^2)\prod_{i=1}^{n}|x-a_i|^{-\frac{N}{n}}
\ee
which belongs to $\hoi$ (since $n\geq 3$). Then,  for $u$ in \eqref{bob} the right hand side in \eqref{impid} vanishes and therefore $u$ satisfies
\be\label{attains}
\frac{\into |\n u|^2 \dx}{\into V u^2 \dx}=\frac{N^2}{n^2}.
\ee
From \eqref{impid} and \eqref{attains} it follows that $\mu^\star(\Omega)=N^2/n^2$  which is attained in $\hoi$ by $u$ in \eqref{bob}.

\noindent{\bf Step III: Optimality and non-attainability of $N^2/n^2$ in \eqref{hardin} in the case $n= 2$.}

Assume that $N^2/4$ is attained by, say, a function $v\in \hoi$. In view of \eqref{impid} we necessary obtain that $v$ must coincide, up to a multiplicative constant, with $u$ in \eqref{bob}. This is a contradiction because the function $u$ in \eqref{bob} does not belong to $\hoi$ (the details are let to the reader). Therefore, the constant $N^2/4$ is not attained.

 In order to prove the optimality of $N^2/4$ i.e.  $\mu^\star(\Omega)= N^2/4$, from \eqref{impid} it suffices to construct a minimizing sequence $\{u_\eps\}_{\eps>0}\subset \hoi$ such that
\begin{equation}\label{googdlimitt}
\lim_{\eps \searrow 0}\frac{\into |\n u_\eps|^2 \dx }{
 \into V u_\eps^2 \dx} =\frac{N^2}{4}.
\end{equation}
Such a sequence is a truncation of $u$ in \eqref{bob} in the neighborhood of the singular poles as follows.

Let $\eps>0$ aimed to be small $(\eps< \min\{1, d/2\})$ and consider the sequence $\{u_\eps\}_{\eps>0}\subset \hoi$ defined by
\be\label{seq}
u_\eps:=\theta_\eps f \prod_{i=1}^{2} |x-a_i|^{-\frac{N}{2}},  \ee
where $f=r^2-|x-x_0|^2$ and the cut-off function $\theta_\eps\in C(\rr^N)$, supported far from the singular poles,  is given by
\begin{equation}\label{teta}
\theta_\eps (x)=\left \{\begin{array}{ll}
                  0, & |x-a_i|\leq \eps^2, \q \forall i\in\{1,  2\},  \\ [5pt]
                  \frac{\log |x-a_i|/\eps^2 }{\log 1/\eps}, & \eps^2 \leq |x-a_i|\leq
                  \eps, \q \forall i\in\{1, 2\}, \\ [5pt]
                 1,&  \textrm{otherwise}.
                             \end{array}\right.
\end{equation}
Then, from \eqref{impid} and \eqref{seq} we obtain
\begin{align}\label{impide}
\into |\n u_\eps|^2 \dx &- \frac{N^2}{4} \into V u_\eps^2 \dx=\into \left|\n\theta_\eps \right|^2 f^2 \prod_{i=1}^{2} |x-a_i|^{-N} \dx,
\end{align}
which is equivalent to
\begin{align}\label{impidediv}
\frac{\into |\n u_\eps|^2 \dx}{\into V  u_\eps^2 \dx}- \frac{N^2}{4} =\frac{\into \left|\n\theta_\eps \right|^2 f^2 \prod_{i=1}^{2} |x-a_i|^{-N} \dx}{\into V  u_\eps^2 \dx}.
\end{align}
We conclude the proof  by showing that the right hand side in \eqref{impidediv} converges to zero as $\eps$ tends to zero.


Indeed, we observe that there exists a constant $C>0$ (independent of $\eps$) such that
\be\label{small}
\into V u_\eps^2 \dx > C>0, \quad \forall \eps>0.
\ee
On the other hand we remark that
\begin{align}\label{boun}
f(x)=\rho(x)(2r -\rho(x))\leq 2r |x-a_i|, \quad \forall x\in B_r(x_0), \ \forall i\in \{1, 2\},
\end{align}
where $\rho(x)=r-|x-x_0|$.
Then, taking into account the support of $\n \theta_\eps$ from \eqref{teta} and \eqref{boun} we successively obtain
\begin{align}\label{cici}
I_1:&=\into \left|\n\theta_\eps \right|^2  f^2 \prod_{i=1}^{2} |x-a_i|^{-N} \dx \nonumber\\
& \leq 4r^2 \sum_{i=1}^{2}\int_{B_\eps(a_i)\setminus B_{\eps^2}(a_i)}|\n \theta_\eps|^2 |x-a_i|^2 \prod_{j=1}^{2} |x-a_j|^{-N} \dx\nonumber\\
&  = 4r^2\sum_{i=1}^{2} \int_{B_\eps(a_i)\setminus B_{\eps^2}(a_i)}
\frac{1}{\log^2(1/\eps)} \prod _{j=1}^{2}
|x-a_j|^{-N}\dx.
\end{align}
Since
\begin{equation}
|x-a_j|\geq \frac{d}{2}, \q \forall x\in B_\eps(a_i), \q \forall j
\neq i, \q \forall i, j\in \{1,  2\},
\end{equation}
from \eqref{cici} we deduce  that
\begin{align}\label{lob4}
I_1 & \leq \frac{ 4r^2\big(
\frac{d}{2}\big)^{-N}}{\log^2(1/\eps)} \sum_{i=1}^{2}\int_{B_\eps(a_i)\setminus
B_{\eps^2}(a_i)}
|x-a_i|^{-N}\dx\nonumber\\
&=\frac{8r^2\omega_N \big( \frac{d}{2}\big)^{-N}}{\log^2(1/\eps)} \int_{\eps^2}^{\eps} r^{-1}\dr.
\end{align}
 From \eqref{lob4} we obtain that for some uniform constant $C>0$ we get
\begin{equation}\label{landau}
I_1\leq \frac{C}{\log (1/\eps)}, \textrm{ as } \eps\rightarrow 0.
\end{equation}
Applying  \eqref{landau}, \eqref{small} and \eqref{impidediv} we finally obtain \eqref{googdlimitt}.
$\hfill\Box$

\subsubsection{The case of the exterior of a ball}\label{secextball} Let us consider  $\Omega=B_{r}^{c}(x_0)=\rr^N\setminus B_r(x_0)$ for some $r>0$ and fix $x_0\in \rr^N$.
The proof is very similar to the one in Subsection \ref{secball} as follows.

\noindent{\bf Step I: Election of $f$ in \eqref{impcond4}.}  Let us consider \be\label{newf}
f(x):=|x-x_0|^2 -r^2.
\ee Remark that $f$ in \eqref{newf} is exactly the function in \eqref{natelec} with the opposite sign. Due to \eqref{cr},   $f$  in \eqref{newf} satisfies \eqref{impcond4} and therefore we deduce that $\mu^\star(B_{r}^{c}(x_0))\geq N^2/n^2$ which leads to \eqref{hardin}.

\noindent{\bf Step II:  Attainability of $N^2/n^2$ in \eqref{hardin} when $N\geq 3$ and $n\geq 3$.}  The attainability  goes similarly  to  the previous situation corresponding to a ball. Indeed,   if the constant $N^2/n^2$ in \eqref{hardin} were attained by some function $u\in \semi$,  from \eqref{impcon} we necessary have (up to a multiplicative constant)
\be\label{crufu}
u= (|x-x_0|^2-r^2)\prod_{i=1}^{n}|x-a_i|^{-\frac{N}{n}}.
\ee
Such $u$ belongs to $\semi$ when $N\geq 3$ and $n\geq 2$. This is true because $|\n u|^2$ is locally integrable near each singularity $a_i$  when  $n\geq3$, whereas the integrability at infinity is ensured by $N\geq 3$.

\noindent{\bf Step III: Optimality of $N^2/n^2$ in \eqref{hardin} for any $N\geq 2$ and $n\geq2$.} In the cases in which $N^2/n^2$ is not attained we have to proceed with an approximation argument similar to the case of a ball  in \eqref{googdlimitt}-\eqref{seq}.

However, the cut-off function $\theta_\eps$ defined in \eqref{teta} requires a different definition at infinity in this present case. Let $\eps>0$ small enough such that
$$B_\eps(a_i)\cap B_\eps(a_j)=\emptyset, \quad  i\neq j, \textrm{ and } \cup_{i=1}^{n} B_\eps(a_i)\subset \Omega. $$
 Next we consider
\begin{equation}\label{teta1}
\theta_\eps (x)=\left \{\begin{array}{ll}
                  0, & |x-a_i|\leq \eps^2, \q \forall i\in\{1, \ldots, n\},  \\ [5pt]
                  \frac{\log |x-a_i|/\eps^2 }{\log 1/\eps}, & \eps^2 \leq |x-a_i|\leq
                  \eps, \q \forall i\in\{1, \ldots, n\}, \\ [5pt]
                 1,&  x\in B_{1/\eps}(x_0)\setminus \cup_{i=1}^{n}
                 B_\eps(a_i),\\[5pt]
                 \frac{\log 1/(\eps^2 |x-x_0| )}{\log 1/\eps}, & x\in B_{1/\eps^2}(x_0)\setminus B_{1/\eps}(x_0),\\[5pt]
               0, & |x-x_0|\geq 1/\eps^2.
                             \end{array}\right.
\end{equation}
and we define the   sequence  $u_\eps:=\theta_\eps f \prod_{i=1}^{n}|x-a_i|^{-N/n}$ for $f$  in \eqref{newf} and $\theta_\eps$ in \eqref{teta1}. Applying \eqref{impcon} for these new weights we get
 \begin{align}\label{impide1}
\into |\n u_\eps|^2 \dx &- \frac{N^2}{n^2} \into V u_\eps^2 \dx=\into \left|\n\theta_\eps \right|^2 f^2 \prod_{i=1}^{n} |x-a_i|^{-2N/n} \dx.
\end{align}
 Remark that $\n \theta_\eps$ is supported in $\cup_{i=1}^{n} \left(B_{\eps}(a_i)\setminus B_{\eps^2}(a_i)\right)\cup  \left(B_{1/\eps^2}(x_0)\setminus B_{1/\eps}(x_0)\right)$.
 Next let us denote by $I_1$ and $I_2$ the integrals of the right hand side in \eqref{impide1} restricted to $\cup_{i=1}^{n}\left(B_{\eps}(a_i)\setminus B_{\eps^2}(a_i)\right)$  and  $ B_{1/\eps^2}(x_0)\setminus B_{1/\eps}(x_0)$ respectively.
Observe that $I_1$ was already computed in the case of the ball when $n=2$ and stands for a  quantity which converges to zero as $\eps$ tends to zero. When $n\geq 3$, $I_1$ tends to zero even more rapidly than in the case when $n=2$.

Now we proceed with  the estimates for $I_2$.   For $\eps>0$ small enough such that $\eps\leq  1/2r$,  it holds
\begin{equation}\label{bounds1}
|x-a_i|\geq \frac{1}{2}|x-x_0|, \q \forall x\in \rr^N \setminus
B_{1/\eps}(x_0), \ \forall i\in \{1, \ldots, n\}.
\end{equation}
Due to \eqref{bounds1} we obtain
\begin{align}\label{secc}
I_2&= \left( \log \frac{1}{\eps} \right)^{-2} \int_{B_{1/\eps^2}(x_0) \setminus
B_{1/\eps}(x_0)}  |x-x_0|^{-2} (|x-x_0|^2-r^2)^2  \prod_{i=1}^{n} |x-a_i|^{-2N/n}
\dx\nonumber\\
& \leq  \left( \log \frac{1}{\eps} \right)^{-2} \int_{B_{1/\eps^2}(x_0)\setminus
B_{1/\eps}(x_0)} |x-x_0|^{2} \prod_{i=1}^{n} \left(\frac{1}{2}|x-x_0|\right)
^{-2N/n}\dx \nonumber\\
 &= \left( \log \frac{1}{\eps} \right)^{-2} 2^{2N}
\int_{B_{1/\eps^2}(x_0)\setminus B_{1/\eps}(x_0)} |x-x_0|^{2-2N} \dx \nonumber\\
&= \left( \log \frac{1}{\eps} \right)^{-2} 2^{2N} \omega_N
\int_{1/\eps}^{1/\eps^2}  s^{1-N} \ds.
\end{align}
From \eqref{secc} we get that
\begin{equation}\label{seec}
I_2\leq\left\{
\begin{array}{ll}
  C\left(\log \frac{1}{\eps}\right)^{-1}, & N=2, \\
  C  \eps^{N-2} \left( \log \frac{1}{\eps}\right)^{-2}, & N\geq 3,
\end{array}\right. \textrm{ as } \eps \to 0,
\end{equation}
for some constant $C>0$ depending only on $N$ and $\omega_N$.
Combining \eqref{seec} with the similar estimates for $I_1$ obtained in the case of a ball, we end up with
\be\label{cioc}
\into \left|\n\theta_\eps \right|^2 f^2 \prod_{i=1}^{n} |x-a_i|^{-2N/n} \dx \to 0, \textrm{ as } \eps \to 0.
\ee
In addition, the uniform estimate in \eqref{small} remains valid in this case, which together with \eqref{cioc} yield to the optimality of $N^2/n^2$.
$\hfill\Box$

\subsubsection{The case of a half-space}\label{sechalfspace}

For simplicity, let us focus on the upper-half space  $\Omega=\rr_{+}^{N}$.
Once we have dealt with Subsections \ref{secball}-\ref{secextball}, the proof of this case is easy and natural.

 The inequality $\mu^\star(\rr_{+}^{N})\geq N^2/n^2$ derives from \eqref{impcon} for
\be\label{fhalf}
f=x_N,
\ee
 which  verifies \eqref{impcond4}.

 Again, if $u\in \semi$ were a minimizer for $N^2/n^2$, in view of \eqref{impcon} we should necessary have
\be\label{coc1}
u=x_N \prod_{i=1}^{n}|x-a_i|^{-\frac{N}{n}}.
\ee
Similarly as in Subsections \ref{secball}-\ref{secextball}, the function $u$ in \eqref{coc1} belongs to $\semi$ for any $n\geq3$.   So the constant $N^2/n^2$ is attained when $n\geq 3$. The optimality  of $N^2/n^2$ can be proven making use of  the minimizing sequence $u_\eps:=\theta_\eps f \prod_{i=1}^{n}|x-a_i|^{-N/n}$,  with $f$  as in \eqref{fhalf} and $\theta_\eps$ verifying
\begin{equation}\label{teta2}
\theta_\eps (x)=\left \{\begin{array}{ll}
                  0, & |x-a_i|\leq \eps^2, \q \forall i\in\{1, \ldots, n\},  \\[5pt]
                  \frac{\log |x-a_i|/\eps^2 }{\log 1/\eps}, & \eps^2 \leq |x-a_i|\leq
                  \eps, \q \forall i\in\{1, \ldots, n\}, \\[5pt]
                 1,&  x\in B_{1/\eps}(0)\setminus \cup_{i=1}^{n}
                 B_\eps(a_i),\\[5pt]
                 \frac{\log 1/(\eps^2 |x| )}{\log 1/\eps}, & x\in B_{1/\eps^2}(0)\setminus B_{1/\eps}(0),\\[5pt]
               0, & |x|\geq 1/\eps^2.
                             \end{array}\right.
\end{equation}
This approximation is relevant for $n=2$ since for any $n\geq 3$ the constant $N^2/n^2$ is attained. 
The computations are similar to those in Subsections \ref{secball}-\ref{secextball}. The details are let to the reader.  $\hfill \Box$

Thus, Theorem \ref{t1} is totally proven.

\subsection{Proof of Theorem \ref{te2}}\label{sect4}

For this proof we need to apply some local results similar to those in Lemma \ref{lema1.int} in the context of interior singularities. We state the following lemma.
\begin{lema}[local results for boundary singularities]\label{lema1}
Assume $\Omega\subset \rr^N$, $N\geq 2$, is a smooth bounded domain such that $a_1, \ldots, a_n\in \Gamma$, $n\geq 2$.
For every $\eps>0$ small enough,  there exists $U_{\eps}$ a neighborhood  of $\cup_{i=1}^{n} \{a_i\}$ in $\overline{\Omega}$ such that
\begin{enumerate}[(i).]
\item\label{1} For any $u\in C_{0}^{\infty}(U_\eps)$ it follows that
\be\label{local.hardy}
\into |\n u|^2\dx\geq \left(\frac{N^2}{4(n-1)}-\eps \right)\into V u^2\dx. \ee
\item\label{2} There exists $u_\eps\in C_{0}^{\infty}(U_\eps)$ such that
\be\label{local.revers}
\into |\n u_\eps|^2 \dx \leq \left(\frac{N^2}{4(n-1)}+\eps\right)\into  V u_\eps^2 \dx.
\ee
\end{enumerate}
\end{lema}
\noindent\textit{Sketch of the proof.} We need to apply the Hardy inequality with boundary singularity \eqref{classical_Hardy_inequality}, locally near each singular pole. More precisely, according to \cite[Lemma 2.1]{MR2847469} for each pole $a_i$ there exists $r_i>0$ small enough (depending on the local geometry of the boundary near the pole $a_i$) such that
\be\label{loc.bound.fall}
\into |\n u|^2\dx \geq \frac{N^2}{4}\into \frac{u^2}{|x-a_i|^2}\dx, \quad \forall u\in C_{0}^{\infty}(\overline{\Omega}\cap B_{r}(a_i)), \quad \forall i,
\ee
and the constant $N^2/4$ is optimal in \eqref{loc.bound.fall}.  Then, the proof follows straightforward as in Lemma \ref{lema1.int}. The details are let to the reader.
$\hfill \Box$

Now we return to the proof of Theorem \ref{te2}. For the simplicity of presentation we start the other way around with the proofs of items \eqref{item6}-\eqref{item8}.

\subsubsection*{Proofs of items \eqref{item6}-\eqref{item8}.} They follow the same lines as the proofs of items \eqref{item1}-\eqref{item3} in Theorem \ref{moregen} by applying Lemma \ref{lema1} instead of Lemma \ref{lema1.int}.

\subsubsection*{Proof of item \eqref{item5}.} The strict lower bound in \eqref{bounds.bond} is a trivial consequence of Theorem \ref{prop1} when $n\geq 3$. Now, let us prove \eqref{bounds.bond} when $n=2$. Indeed, let us assume by absurd that the lower bound in \eqref{bounds.bond} is not strict when $n=2$. This implies that $\mu^\star(\Omega; a_1, a_2)=(N-2)^2/4$ where $a_1, a_2$ are two singular poles localized on the boundary.    Since $\mu^\star(\Omega)=(N-2)^2/4< N^2/4$, according to item \eqref{item8}, we deduce that  $\mu^\star(\Omega)=(N-2)^2/4$ is attained by a function $u\in \hoi$. Applying \eqref{optint} we necessary obtain  $u=C\prod_{i=1}^{2}|x-a_i|^{-(N-2)/2}$, for some nontrivial constant $C\in \rr$.  Contradiction, because such $u$ does not belong to $\hoi$.

Finally,  the upper bound in \eqref{bounds.bond} is a consequence of item \eqref{2} in Lemma \ref{lema1}.

Thus, the proof of Theorem \ref{te2} is finished.  $\hfill$ $\Box$

\section{Other remarks, comments and open questions}\label{com}

$\bullet$ We have seen in bounded domains with interior singularities that there is a gap between $\mu^\star(\Omega)$ and  $\mu^\star(\rr^N)=(N-2)^2/n^2$ when $n\geq 3$. It would be very interesting  to determine that gap explicitly.

$\bullet$ We have shown  that Theorem \ref{t1} provides optimal results for particular geometries like balls, exterior of balls or half-spaces. Furthermore, we may stress the question  of determining more general classes of domains for which Theorem \ref{t1}  applies. For instance, we may ask weather Theorem \ref{t1} is valid for convex domains. A positive answer to this latter question does not seem to be trivial at all. Indeed,  the optimality of the Hardy inequality with one singular potential obtained in balls and half-spaces is enough to extend similar results to more general domains like convex domains. This is due to the comparison arguments which make use of the anti-monotonicity properties of the Hardy constant $\mu^\star(\Omega)$ with respect to domains inclusions.

Generally,  the anti-monotonicity property cannot be used in an efficient way for multipolar potentials. Therefore, in order to check wether  $\mu^\star(\Omega)=N^2/n^2$  for convex domains, it suffices to build weights  $f$ satisfying \eqref{impcond4}.  The existence of such weighted functions in any other domains would be also enough to obtain $\mu^\star(\Omega)=N^2/n^2$.

In particular, there are non-convex domains for which $\mu^\star(\Omega)=N^2/n^2$. Indeed, in view of the proof of Theorem \ref{t1} this is true in any domain included in a ball, touching the boundary of the ball at the singular poles.


$\bullet$ Moreover, we may wonder if the family of constants $\{c_\mu\}_{\mu>0}$  in Theorem \ref{moregen} is uniformly bounded as $\mu$ tends to $\mu_\omega^\star(\Omega)=(N-2)^2/(4n-4)$ so that we can obtain the validity of Theorem \ref{moregen} even in the critical case $\mu=\mu_\omega^\star(\Omega)$. Equivalently, this would ask to the attainability of $\mu_\omega^\star(\Omega)$ in \eqref{weakoptimalHardy}.  The same questions stand in the case of boundary singularities where $\mu_\omega^\star(\Omega)=N^2/(4n-4)$ as shown in Theorem \ref{te2}.  Unfortunately, our approaches do not allow to say anything wether these critical values are attained or not.


%
%

%
%

$\bullet$  Finally,  we have obtained the Hardy inequality  with the optimal constant $\mu^\star(\Omega)=N^2/n^2$ in various domains with different geometries like the half-space, the ball, the exterior of a ball, etc.. In proving so, we had to make an adequate election for the weight $f$ in \eqref{impcon}.  Previous examples suggest that $f$ must be chosen in terms of the implicit equation of the boundary $\Gamma$ in the neighborhood of the singular poles $a_i$.

It is important to remark that this is not the case of an ellipse. More precisely, assume $\Omega\subset \rr^2$ is the interior of an ellipse given by
$$\mathcal{E}: \quad \frac{x_1^2}{a^2}+\frac{x_2^2}{b^2}\leq 1,$$
with the singular poles $a_i=(a_i^1, a_i^2)$ lying on its boundary.

Moreover, let us consider the weight function induced by the implicit equation of the ellipse, that is $f=1-\frac{x_1^2}{a^2}+\frac{x_2^2}{b^2}$ and satisfies $f>0$ in $\mathcal{E}$ and $f=0$ on the boundary of $\mathcal{E}$.
For such $f$ we evaluate the expression $S(f, a_i)$ in \eqref{impcond4}:
\begin{align}\label{nonconst}
-\D f |x-a_i|^2 + &2 N \n f \cdot (x-a_i) - 2 N f=\nonumber\\
&=\frac{2(b^2-a^2)\Big( (x_2-a_i^2)^2 -(x_1-a_i^1)^2\Big)}{a^2 b^2}.
\end{align}
 Observe that quantity \eqref{nonconst} does not have a constant sign in $\mathcal{E}$ so the third condition in \eqref{impcond4} is not verified for such $f$.  In consequence, we cannot make any statement about the Hardy constant $\mu^\star(\Omega)$ in this case.

$\bullet$ To conclude, one of the main further challenges is to extend our analysis to multipolar potentials of the form
 \be\label{other}
 V=\sum_{i=1}^{n}\frac{1}{|x-a_i|^2},
 \ee
for which, to our knowledge, the corresponding Hardy constants are not known for any particular domains. This problem has been intensively studied in the literature quoted below, but optimal results are still to be obtained. We believe that the  optimal results of our paper  could be a hint in order to handle other types of multi-singular potentials such as \eqref{other}.

\noindent {\bf Acknowledgements.}

The author wishes to thank the referee for his  valuable remarks and helpful comments which leaded to a better redaction of the paper.

This work was partially supported by the both grants of the Ministry of National Edu\-cation, CNCS-UEFISCDI Romania, project PN-II-ID-PCE-2012-4-0021  and project  PN-II-ID-PCE-2011-3-0075, and the Grant  MTM2011-29306-C02-00 of the MICINN (Spain).


\begin{thebibliography}{99}


\bibitem{agmon} S. Agmon, \textit{Bounds on exponential decay of eigenfunctions of
              {S}chr\"odinger operators}, ``Schr\"odinger operators" ({C}omo, 1984), pp. 1--38, Lecture Notes in Math. 1159, Springer, Berlin, 1985.

\bibitem{MR0374628}  W. Allegretto, \textit{On the equivalence of two types of oscillation for elliptic operators} Pacific J. Math. {\bf 55} (1974),  319--328.

 \bibitem{MR2083152} A. Balinsky, A. Laptev, and A. V. Sobolev, \textit{Generalized Hardy inequality for the magnetic Dirichlet forms},
J. Statist. Phys. {\bf 116} (2004), no. 1-4, 507--521.

 \bibitem{MR2379440} R. Bosi, J. Dolbeault, and M. J. Esteban, \textit{Estimates for the optimal constants in multipolar Hardy inequalities for Schr\"{o}dinger and Dirac operators}, Commun. Pure Appl. Anal. {\bf 7} (2008), no. 3, 533--562.

\bibitem{MR1655516} H. Brezis and M. Marcus,  \textit{Hardy's inequalities revisited},  Dedicated to Ennio De Giorgi. Ann. Scuola Norm. Sup. Pisa Cl. Sci. (4) {bf 25} (1997), no. 1-2, 217--237 (1998).

 \bibitem{MR1605678}  H. Brezis and J. L. V\'{a}zquez, \textit{Blow-up solutions of some nonlinear elliptic problems}, Rev. Mat. Univ.
Complut. Madrid {\bf 10} (1997), no. 2, 443--469.

\bibitem{caldiroli2} P. Caldiroli and R. Musina, \textit{On a class of two-dimensional singular elliptic problems}, Proc. Roy. Soc. Edinburgh Sect. A {\bf 131} (2001), no. 3, 479--497.

\bibitem{cras1} C. Cazacu, \textit{On Hardy inequalities with singularities on the boundary}, C. R. Math. Acad. Sci. Paris {\bf 349} (2011), no. 5-6, 273--277.

\bibitem{cristiheat} C. Cazacu, \textit{Controllability of the {H}eat {E}quation with an  {I}nverse-{S}quare {P}otential {L}ocalized on the {B}oundary}, SIAM J. Control Optim. {\bf 52} (2014), no. 4, 2055--2089.

 \bibitem{cristi1} C. Cazacu and E. Zuazua, \textit{Improved Multipolar {H}ardy Inequalities}, Studies in Phase Space Analysis of PDEs, Progress in Nonlinear Differential Equations and Their Applications, vol. {\bf  84}, Birkh\"{a}user, New York, 2013, 37--52.

 \bibitem{MR1366614} E. B. Davies, \textit{The Hardy constant}, Quart. J. Math. Oxford Ser. (2) {\bf 46} (1995), no. 184, 417--431.

 \bibitem{MR1747888} E. B. Davies, \textit{A review of Hardy inequalities}, The Maz'ya anniversary collection, Vol. {\bf 2} (Rostock, 1998), Oper.
Theory Adv. Appl., vol. {\bf 110}, Birkh\"{a}user, Basel, 1999,  55--67.

\bibitem{DD} J. D\'{a}vila and L. Dupaigne, \textit{Comparison results for PDEs with a singular potential}, Proc. Roy. Soc. Edinburgh Sect. A {\bf 133} (2003), no. 1, 61--83.

\bibitem{devyver} B. Devyver,  M. Fraas and Y. Pinchover, \textit{Optimal Hardy weight for second-order elliptic operator: an
              answer to a problem of {A}gmon}, J. Funct. Anal. {\bf 266} (2014), no. 7, 4422--4489.

\bibitem{MR2847469} M. M. Fall, \textit{A note on Hardy's inequalities with boundary singularities} Nonlinear Anal. {\bf 75} (2012), no. 2, 951--963.

\bibitem{MR2966112} M. M.  Fall and R. Musina, \textit{Hardy-Poincar\'{e} inequalities with boundary singularities},  Proc. Roy. Soc. Edinburgh Sect. A {\bf 142} (2012), no. 4, 769--786.

 \bibitem{fefe} C. Fefferman, \textit{The uncertainty principle}, Bull. Amer. Math. Soc. (N. S.) {\bf 9} (1983), no. 2, 129--206.

 \bibitem{MR2735078} V. Felli, A. Ferrero, and S. Terracini, \textit{Asymptotic behavior of solutions to Schr\"{o}dinger equations near an
isolated singularity of the electromagnetic potential}, J. Eur. Math. Soc.  {\bf 13} (2011), no. 1, 119--174.

 \bibitem{MR2514383} V. Felli, E. M. Marchini, and S. Terracini, \textit{On Schr\"{o}dinger operators with multisingular inverse-square
anisotropic potentials}, Indiana Univ. Math. J. {\bf 58} (2009), no. 2, 617--676.

 \bibitem{MR2241305} V. Felli and S. Terracini, \textit{Nonlinear Schr\"{o}dinger equations with symmetric multi-polar potentials}, Calc.
Var. Partial Differential Equations {\bf 27} (2006), no. 1, 25--58.

\bibitem{hardy-polya} G. H. Hardy, J. E.  Littlewood G. and P{\'o}lya, \textit{Inequalities}, Reprint of the 1952 edition, Cambridge University Press, Cambridge, 1988.

 \bibitem{MR2679028} D. Krej{\v{c}}i{\v{r}}{\'{\i}}k and E. Zuazua, \textit{The Hardy inequality and the heat equation in twisted tubes}, J. Math. Pures
Appl. (9) {\bf 94} (2010), no. 3, 277--303.

 \bibitem{laptev} A. Laptev and T. Weidl, \textit{Hardy inequalities for magnetic Dirichlet forms}, Oper. Theory: Advances and
Applications {\bf 108} (1999), 299--305.

\bibitem{MR0342829} J. Piepenbrink, \textit{Nonoscillatory elliptic equations},  J. Differential Equations {\bf 15} (1974), 541--550.

 \bibitem{tintarevpinchover} Y. Pinchover and K. Tintarev, \textit{Existence of minimizers for Schr\"{o}dinger operators under domain perturbations with application to Hardy's inequality}, Indiana Univ. Math. J. {\bf 54} (2005), no. 4, 1061--1074.

\bibitem{hiddenenergy}
J. L. V\'{a}zquez and N. B. Zographopoulos,
  \textit{Functional aspects of the Hardy inequality.
Appearance of a hidden energy}, Discrete Contin. Dyn. Syst.-A {\bf 33} (2013), no. 11-12, 5457--5491.
\end{thebibliography}

\end{document}